\documentclass[12pt]{amsart}
\usepackage{amsmath,amsthm,amsfonts,amscd,amssymb,eucal,latexsym,mathrsfs, appendix,verbatim,nameref}
\usepackage[numbers,sort&compress]{natbib}
\usepackage[all,cmtip]{xy}
\newtheorem{theorem}{Theorem}[section]

\newtheorem{lemma}[theorem]{Lemma}
\newtheorem{proposition}[theorem]{Proposition}
\theoremstyle{definition}
\newtheorem{definition}[theorem]{Definition}
\newtheorem{remark}[theorem]{Remark}

\newtheorem{example}[theorem]{Example}
\newtheorem{question}[theorem]{Question}

\newcommand{\Complex}{\mathbb{C}}
\allowdisplaybreaks

\begin{document}
\title{Faithful compact quantum group actions on
connected compact metrizable spaces}
\author{Huichi Huang}
\address{Department of Mathematics, SUNY at Buffalo, Buffalo, NY 14260, U.S.A.}
\email{hh43@buffalo.edu}
\date{January 29, 2012}
\begin{abstract}
We construct faithful actions of quantum permutation groups on connected compact metrizable spaces. This disproves a conjecture of Goswami.
\end{abstract}

\maketitle
\section{Introduction}

 Compact quantum groups were introduced by Woronowicz in \cite{Woronowicz1987, Woronowicz1998}. They are noncommutative analogues of compact groups. Among all literatures related to compact quantum groups, one particularly interesting topic is the compact quantum group actions on commutative or non-commutative unital C*-algebras  (from the viewpoint of non-commutative topology, that means actions on commutative or non-commutative compact spaces). The actions of compact quantum groups are the natural generalizations of actions of compact groups. It was Podle\'{s} who first formulated the concept of compact quantum group actions, then established some basic properties~\cite{Podles1995}. Later, Wang introduced the quantum permutation groups~ \cite{Wang1998} and showed that they are the universal compact quantum groups acting on finite spaces. After that, many interesting actions are studied (see~\cite{Podles1995, Banica2005,Banica200502,Goswami2009,BG2009,BGS2011,Goswami2011} and the references therein). But so far, all known (commutative) compact spaces admitting genuine faithful compact quantum group actions are disconnected. In~\cite{Goswami2011}, Goswami showed that there is no genuine faithful quantum isometric action of compact quantum groups on the Riemannian manifold $G/T$ where $G$ is a compact, semisimple, centre-less, connected Lie group with a maximal torus $T$  and conjectured that the quantum permutations on (disconnected) finite sets are the only possible faithful actions of genuine compact quantum groups on classical spaces. In this paper, we construct faithful actions of quantum permutation groups on connected compact metrizable spaces and disprove Goswami's conjecture.

The paper is organized as follows. In the next section we recall some basic definitions and terminologies related to compact quantum groups and their actions. Then in section 3, we construct faithful quantum permutation group actions on connected compact metrizable spaces.

\section{Preliminaries}

In this section, we recall some definitions about compact quantum groups. See \cite{Podles1995,Wang1995,Wang1998,Woronowicz1987,Woronowicz1998} for more details. Throughout this paper, the notation $A\otimes B$ for two unital C*-algebras $A$ and $B$ stands for the minimal tensor product of $A$ and $B$. For a $*$-homomorphism $\beta:B\rightarrow B\otimes A$, use $\beta(B)(1\otimes A)$ to denote the linear span of the set $\{\beta(b)(1_B\otimes a)|b\in B, \,a\in A\}$ and  $\beta(B)(B\otimes1)$ to denote the linear span of the set $\{\beta(b_1)(b_2\otimes 1_A)|b_1, b_2\in B\}$. Denote by $\Complex$ the set of complex numbers. For a compact Hausdorff space $X$ and a unital C*-algebra $A$, denote by $C(X,A)$ the C*-algebra of continuous functions mapping from $X$ to $A$. Especially, when $A=\Complex$, we write $C(X,\Complex)$ as $C(X)$. For $x\in X$, use $ev_x$ to denote the evaluation functional on $C(X)$ at the point $x$.

\begin{definition} [Definition 1.1 in \cite{Woronowicz1998}]
A compact quantum group is a unital C*-algebra $A$ together with a unital $*$-homomorphism $\Delta: A\rightarrow A\otimes A$ such that
\begin{enumerate}
\item $(\Delta\otimes id)\Delta=(id\otimes \Delta)\Delta$;

\item $\Delta(A)(1\otimes A)$ and $\Delta(A)(A\otimes1)$ are dense in $A\otimes A$.
\end{enumerate}
\end{definition}

If $A$ is non-commutative, we say that $(A,\Delta)$ is a {\it genuine} compact quantum group.
We say a unital C*-subalgebra $Q$ of $A$ is a {\it compact quantum quotient group} of $(A,\Delta)$ if  $\Delta(Q)\subseteq Q\otimes Q$, and $\Delta(Q)(1\otimes Q)$ and $\Delta(Q)(Q\otimes1)$ are dense in $Q\otimes Q$. That is, $(Q,\Delta|_{Q})$ is a compact quantum group~\cite[Definition 2.9]{Wang1995}. If $Q\neq A$, we call $Q$ a {\it proper compact quantum quotient group}.

\begin{definition}[Definition 1.4 in \cite{Podles1995}]
An action of a compact quantum group $(A,\Delta)$ on a unital C*-algebra $B$ is a unital $*$-homomorphism $\alpha:B\rightarrow B\otimes A$ satisfying that
\begin{enumerate}
\item $(\alpha\otimes id)\alpha=(id\otimes \Delta)\alpha$;

\item $\alpha(B)(1\otimes A)$ is dense in $B\otimes A$.
\end{enumerate}
\end{definition}

When $A$ is non-commutative, we call $\alpha$ a genuine compact quantum group action.
We say that $\alpha$ is {\it faithful} if there is no proper compact quantum quotient group $Q$ of $A$ such that $\alpha$ induces an action $\alpha_q$ of $Q$ on $B$ satisfying $\alpha(b)=\alpha_q(b)$ for all $b$ in $B$~\cite[Definition 2.4]{Wang1998}. An action $\alpha$ is called {\it ergodic} if $\{b\in B|\alpha(b)=b\otimes 1_A\}=\Complex 1_B$.
If $A$ acts on $C(X)$ for a compact Hausdorff space $X$, we say that $A$ acts on $X$.

For any positive integer $n$, let $A_n$ be the universal C*-algebra generated by $a_{ij}$ for $1\leq i,j\leq n$ under the relations $a_{ij}^*=a_{ij}=a_{ij}^2$ and $\sum_{i=1}^n a_{ij}=\sum_{j=1}^n a_{ij}=1$. Let $\Delta_n: A_n\rightarrow A_n\otimes A_n$ be the $*$-homomorphism satisfying that $\Delta_n(a_{ij})=\sum_{k=1}^n a_{ik}\otimes a_{kj}$. Then $(A_n,\Delta_n)$ is a compact quantum group and is called a {\it quantum permutation group}~\cite[Theorem 3.1]{Wang1998}. Moreover, the quantum permutation group $A_n$ is a genuine quantum group when $n\geq 4$~\cite[The example before Theorem 6.2]{Wang1999}.

Let $X_n=\{x_1,x_2,\ldots,x_n\}$ be the finite space with $n$ points. Define $e_i$ for $1\leq i\leq n$ to be the function on $X_n$ such that $e_i(x_j)=\delta_{ij}$ for $1\leq j\leq n$. There is a compact quantum group action  $\alpha_n: C(X_n)\rightarrow C(X_n)\otimes A_n$  such that
$\alpha_n(e_i)=\sum_{k=1}^n e_k\otimes a_{ki}$ for all $1\leq i \leq n$~\cite[Theorem 3.1]{Wang1998}.

\section{Main Results}
 Let $Y$ be a compact Hausdorff space.
\begin{lemma}
 There exists an action $\alpha$ of the quantum permutation group $A_n$ on $X_n\times Y$  given by $\alpha(e_i\otimes f)=\sum_{k=1}^n e_k\otimes f\otimes a_{ki}$ for all $1\leq i\leq n$ and $f\in C(Y)$.
\end{lemma}

\begin{proof}
The quantum permutation group $A_n$ acts on $X_n$ by $\alpha_n$. Note that $\Complex$ is also a compact quantum group. We denote the trivial action of $\Complex$ on $Y$ by $\sigma$. Then $A_n\otimes \Complex$ is also a compact quantum group and acts on $X_n\times Y$~\cite[Theorem 2.1]{Wang1993}. This action, denoted by $\alpha$, is defined by $\alpha:=\sigma_{23}(\alpha_n\otimes\sigma)$, where $\sigma_{23}:C(X_n)\otimes A_n\otimes C(Y)\otimes \Complex \rightarrow C(X_n)\otimes C(Y)\otimes A_n\otimes \Complex$ denotes the operator flipping the 2nd and 3rd components. Note that $A_n\otimes\Complex\cong A_n$. Therefore, we can say that $A_n$ acts on $X_n\times Y$ by $\alpha$. Since $\alpha_n(e_i)=\sum_{k=1}^n e_k\otimes a_{ki}$ and $\sigma(f)=f\otimes 1$, we have that $\alpha(e_i\otimes f)=\sum_{k=1}^n e_k\otimes f\otimes a_{ki}$ for all $1\leq i\leq n$ and $f\in C(Y)$.
\end{proof}

Let $Y_1$ be a closed subset of $Y$. We define an equivalence relation $\sim$ on $X_n\times Y$ as follows.
For $y',y''$ in $Y$ and $x', x''$ in $X_n$, two points $(x',y')$ and $(x'',y'')$ in $X_n\times Y$ are equivalent if one of the following is true:
\begin{enumerate}
\item $y'=y''\in Y_1$;

\item $y'=y''$ and $x'=x''$.
\end{enumerate}

\begin{lemma}\label{compact Hausdorff lemma}
The quotient space $X_n\times Y/\sim$ is compact and Hausdorff.
\end{lemma}

\begin{proof}
For convenience, denote $X_n\times Y$ by $Z$. The compactness of $Z/\sim$ follows from the compactness of $Z$. To show that $Z/\sim$ is Hausdorff, it suffices to show that the subset $R:=\{(z_1, z_2)\in Z^2|z_1\sim z_2 \}$ of $Z^2$ is closed~\cite[Theorem 8.2]{Rotman1988}. Let $(z_1, z_2)$ be in $Z^2\backslash R$. Thus $z_1\nsim z_2$. Use $(x',y')$ and $(x'',y'')$ to denote $z_1$ and $z_2$ respectively.

Case 1. If $y'\neq y''$, then there exist two open subsets $U$ and $V$ of $Y$ such that $y'\in U$, $y''\in V$ and $U\cap V=\emptyset$. Thus $(X_n\times U)\times (X_n\times V)$ is an open neighborhood of  $(z_1, z_2)$ and is disjoint with $R$.

Case 2. If $y'=y''$, then $x'\neq x''$, and $y'\notin Y_1$. Since $Y_1$ is closed and $Y$ is compact Hausdorff, there exists an open subset $U$ of $Y$ containing $y'$ and $U$ is disjoint with $Y_1$. Consequently $(\{x'\}\times U)\times (\{x''\}\times U)$ is an open neighborhood of  $(z_1, z_2)$ and is disjoint with $R$.

Combining Case 1 and Case 2, we obtain that $Z^2\backslash R$ is open. Hence $R$ is closed and $Z/\sim$ is Hausdorff.
\end{proof}

\begin{lemma}\label{technical lemma}
If an element $F$ of $C(X_n\times Y)$ satisfies that $F(x_i,y')=F(x_j,y')$ for some $y'$ in $Y$ and all $1\leq i,j \leq n$, then $\alpha(F)(x_k,y')=F(x_j,y')1_{A_n}$ for all $1\leq j,k \leq n$.
\end{lemma}
\begin{proof}
Note that $F$ can be written as  $\sum_{i=1}^n e_i\otimes f_i$ where $f_1,...,f_n$ are in $C(Y)$. If $F(x_i,y')=F(x_j,y')$, then $f_i(y')=f_j(y')$. We obtain that
\begin{equation*}
\begin{split}
\alpha(F)(x_k,y)
&=(ev_k\otimes ev_y\otimes id)\alpha(\sum_{i=1}^n e_i\otimes f_i)\\
&=\sum_{i=1}^n (ev_k\otimes ev_y\otimes id)\alpha(e_i\otimes f_i)\\
&=\sum_{i=1}^n \sum_{l=1}^n (ev_k\otimes ev_y\otimes id)(e_l\otimes f_i\otimes a_{li})\\
&=\sum_{i=1}^n f_i(y)a_{ki} \notag
\end{split}
\end{equation*}
for any $y\in Y$ and $1\leq k\leq n$.
Since $f_i(y')=f_j(y')$ for all $1\leq i,j\leq n$, and $\sum_{i=1}^n a_{ki}=1_{A_n}$ for all $1\leq k \leq n$, we get

$$\alpha(F)(x_k,y')=\sum_{i=1}^n f_i(y')a_{ki}=f_j(y')\sum_{i=1}^n a_{ki}=f_j(y')1_{A_n}=F(x_j,y')1_{A_n}$$ for all $1\leq j,k\leq n$. This completes the proof.
\end{proof}

Note that $C(X_n\times Y/\sim)$ is a C*-subalgebra of $C(X_n\times Y)$.

\begin{proposition}\label{p1}
When the action $\alpha$ is restricted on $C(X_n\times Y/\sim)$, it induces an action $\widetilde{\alpha}$ of $A_n$ on $X_n\times Y/\sim$.
\end{proposition}

\begin{proof}
 We first prove the following:
\begin{equation}\label{eq:inclusion}
\alpha(C(X_n\times Y/\sim))\subseteq C(X_n\times Y/\sim)\otimes A_n.
\end{equation}

 Since $C(X_n\times Y/\sim)\otimes A_n\cong C(X_n\times Y/\sim,A_n)$ and $C(X_n\times Y)\otimes A_n\cong C(X_n\times Y,A_n)$, an element $c$ of $C(X_n\times Y)\otimes A_n$ belongs to $C(X_n\times Y/\sim)\otimes A_n$ if and only if
  $(ev_k\otimes ev_y\otimes id)(c)=(ev_l\otimes ev_y\otimes id)(c)$ for all $1\leq k,l\leq n$ and $y\in Y_1$.

 Therefore, to prove \eqref{eq:inclusion}, it suffices to show that
 \begin{equation*}
 (ev_k\otimes ev_y\otimes id)\alpha(F)=(ev_l\otimes ev_y\otimes id)\alpha(F)
 \end{equation*}
 for all $1\leq k,l\leq n$, $y\in Y_1$  and $F$ in $C(X_n\times Y/\sim)$.

Let $F$ be in $C(X_n\times Y/\sim)$. Then $F$ can be written as $\sum_{i=1}^n e_i\otimes f_i$ for $f_i\in C(Y)$ satisfying that $f_i(y)=f_j(y)$ for all $1\leq i,j\leq n$ and $y\in Y_1$. By Lemma~\ref{technical lemma}, we have
$$(ev_k\otimes ev_y\otimes id)\alpha(\sum_{i=1}^n e_i\otimes f_i)=f_j(y)1_{A_n}=(ev_l\otimes ev_y\otimes id)\alpha(\sum_{i=1}^n e_i\otimes f_i)$$ for all $y\in Y_1$ and $1\leq j,k,l\leq n$. This proves ~\eqref{eq:inclusion}.

Next we show that
$\alpha(C(X_n\times Y/\sim))(1\otimes A_n)$ is dense in $C(X_n\times Y/\sim)\otimes A_n$.

It is enough to show that $F\otimes a$ is in the closure of $\alpha(C(X_n\times Y/\sim))(1\otimes A_n)$  for all $F$ in $C(X_n\times Y/\sim)$ and $a$ in $A_n$. Denote $F\otimes a-\alpha(F)(1\otimes a)$ by $G$. Note that $F$ can be written as $\sum_{i=1}^n e_i\otimes f_i$ for $f_i\in C(Y)$ satisfying that $f_i(y)=f_j(y)$ for all $1\leq i,j\leq n$ and $y\in Y_1$. By Lemma~\ref{technical lemma}, we have $\alpha(F)(x_i,y)=F(x_j,y)1_{A_n}$ for all $x_i,x_j$ in $X_n$ and $y$ in $Y_1$. Thus $G|_{X_n\times Y_1}=0$. For arbitrary $\varepsilon>0$, let $U$ be an open subset of $Y$ containing $Y_1$ and satisfying that $\|G(x_i,y)\|<\varepsilon$ for all $(x_i,y)\in X_n\times U$. By Urysohn's Lemma, there exists an $f$ in $C(Y)$, such that $f|_{Y_1}=0$, $f|_{Y\backslash U}=1$ and $0\leq f \leq 1$. Denote $1\otimes f\in C(X_n\times Y)$ by $H_{\varepsilon}$. Then $H_{\varepsilon}|_{X_n\times Y_1}=0$ and $H_{\varepsilon}(x_i,y)=1$ for all $x_i$ in $X_n$ and $y$ in $Y\backslash U$. It follows from Lemma~\ref{technical lemma} that $(\alpha(H_{\varepsilon})G-G)(x_i,y)=0$ for all $x_i\in X_n$ and $y\in Y\backslash U$. Since $0\leq H_{\varepsilon}\leq 1$, for $(x_i,y)\in X_n\times U$, we have
$$\|(\alpha(H_{\varepsilon})G-G)(x_i,y)\|\leq \|\alpha(H_{\varepsilon})-1\|\|G(x_i,y)\|<\varepsilon.$$
Hence $\|\alpha(H_{\varepsilon})G-G\|<\varepsilon$.
Moreover, since $\alpha$ is an action of $A_n$ on $X_n\times Y$, we have that $\alpha(C(X_n\times Y))(1\otimes A_n)$ is dense in $C(X_n\times Y)\otimes A_n$. So there exist $F_i\in C(X_n\times Y)$ and $a_i\in A_n$ for $1\leq i\leq m$ where $m$ is a positive integer such that $\|G-\sum_{i=1}^m\alpha(F_i)(1\otimes a_i)\|<\varepsilon$. It follows from $0\leq H_{\varepsilon}\leq 1$ that $\|\alpha(H_{\varepsilon})G-\alpha(H_{\varepsilon})\sum_{i=1}^m\alpha(F_i)(1\otimes a_i)\|<\varepsilon$. Hence
\begin{align}
\begin{split}
&\|F\otimes a-\alpha(F)(1\otimes a)-\sum_{i=1}^m\alpha(H_{\varepsilon}F_i)(1\otimes a_i)\| \\
&=\|G-\sum_{i=1}^m\alpha(H_{\varepsilon}F_i)(1\otimes a_i)\| \\
&\leq\|G-\alpha(H_{\varepsilon})G\|+\|\alpha(H_{\varepsilon})G-\sum_{i=1}^m\alpha(H_{\varepsilon}F_i)(1\otimes a_i)\|<2\varepsilon. \notag
\end{split}
\end{align}
Note that $H_{\varepsilon}F_i|_{X_n\times Y_1}=0$ for all $1\leq i \leq m$. Thus $H_{\varepsilon}F_i$ is in $C(X_n\times Y/\sim)$ for all $1\leq i \leq m$.
It follows that $\alpha(F)(1\otimes a)+\sum_{i=1}^m\alpha(H_{\varepsilon}F_i)(1\otimes a_i)$ is in $\alpha(C(X_n\times Y/\sim))(1\otimes A_n)$. Since $\varepsilon>0$ is arbitrary, we conclude that $F\otimes a$ is in the closure of $\alpha(C(X_n\times Y/\sim))(1\otimes A_n)$. This completes the proof.
\end{proof}

\begin{theorem}\label{main theorem}
If $Y_1\neq Y$, the action $\widetilde{\alpha}$ of $A_n$ on $X_n\times Y/\sim$ is faithful.
\end{theorem}

\begin{proof}
Suppose $Y_1\neq Y$. Take a point $y_0$ in $Y$ but not in $Y_1$. Since $Y$ is compact Hausdorff, there exists
$f \in C(Y)$ such that $f(y_0)=1$ and $f|_{Y_1}=0$. Note that $e_i\otimes f$ is in $C(X_n\times Y/\sim)$ for any $1\leq i\leq n$. Suppose $Q$ is a compact quantum quotient group of $A_n$ such that $\widetilde{\alpha}$ is an action of $Q$  on $X_n\times Y/\sim$. Then for any $1\leq k\leq n$,
\begin{align*}
(ev_k\otimes ev_{y_0}\otimes id)\alpha(e_i\otimes f)
&=(ev_k\otimes ev_{y_0}\otimes id)(\sum_{l=1}^n e_l\otimes f\otimes a_{li})\\ \notag
&=f(y_0)a_{ki}=a_{ki}
\end{align*}
is in $Q$. Since $i$ is arbitrarily chosen,  we get $a_{ki}\in Q$ for any $1\leq k,i\leq n$. Thus $Q=A_n$. Therefore $\widetilde{\alpha}$ is faithful.
\end{proof}

\begin{proposition}\label{not ergodic}
If $Y$ contains at least two points, the action $\widetilde{\alpha}$ is not ergpdic.
\end{proposition}
\begin{proof}
Since $Y$ consists of at least two points, there exist a non constant function $f\in C(Y)$. Then $1\otimes f$ is in $C(X_n\times Y\backslash\sim)$ and not constant. Also $$\widetilde{\alpha}(1\otimes f)=1\otimes f\otimes 1.$$ This shows that $\widetilde{\alpha}$ is not ergodic.
\end{proof}

\begin{proposition}\label{connectedness}
If $Y$ is connected and $Y_1$ is nonempty, then $X_n\times Y/\sim$ is connected.
\end{proposition}
\begin{proof}
As before, denote $X_n\times Y$ by $Z$. Take any nonempty closed and open subset $U$ of $Z/\sim$. Denote by $\pi$ the quotient map from $Z$ onto $Z/\sim$. It follows that $\pi^{-1}(U)$ is a nonempty, closed and open subset of $Z$. Since $X_n$ is finite, we obtain that $\pi^{-1}(U)=\bigcup_{x_i\in X'}\{x_i\}\times A_i$ where $X'$ is a nonempty subset of $X_n$, and every $A_i$ is a nonempty closed and open subset of $Y$. Since $Y$ is connected, we have $A_i=Y$ for all $x_i\in X'$. Take $y\in Y_1$ and $x_i\in X'$. Let $x_j\in X_n$. Then $\pi(x_j,y)=\pi(x_i,y)\in U$. Thus $x_j$ is in $X'$. Therefore $X'=X_n$ and $U=Z/\sim$. So $Z/\sim$ is connected.
\end{proof}

By Theorem~\ref{main theorem} and Proposition~\ref{connectedness}, if we take a nonempty proper closed subset $Y_1$ of a connected compact Hausdorff space $Y$, then we get a faithful action of $A_n$ on a compact connected space $X_n\times Y/\sim$. To be more specific, we list some examples of $X_n\times Y/\sim$.

\begin{example}\label{example}\
\begin{enumerate}
\item If $Y=[0,1]$ and $Y_1=\{0\}$, then $X_n\times Y/\sim$ is a wedge sum of n unit intervals by identifying $(x_i,0)$ to a single point for all $1\leq i \leq n$. In this case $X_n\times Y/\sim$ is a contractible compact metrizable space.

\item If $Y=S^1$ is a circle, and $Y_1=\{y_0\}$ for some point $y_0$ in $S^1$, then $X_n\times S^1/\sim$ will be the n circles touching at a point, which is a  connected compact metrizable space whose fundamental group is the free group with n generators.
\end{enumerate}
\end{example}

\begin{remark}\label{r1}
By Theorem~\ref{main theorem}, the quantum permutation group $A_n$ can act on the spaces in Example~\ref{example} faithfully. When $n\geq4$, this gives us faithful genuine compact quantum group actions on connected compact metrizable spaces. This disproves the conjecture of Goswami~\cite{Goswami2011} mentioned in the introduction. However, Proposition~\ref{not ergodic} tells us that these faithful genuine compact quantum actions on compact connected spaces are not ergodic. For this reason, we ask the following question:
\end{remark}

\begin{question}
Are there any faithful ergodic genuine quantum group actions on compact connected spaces?
\end{question}

\section*{Acknowledgement}
The author is grateful to Professor Hanfeng Li for his long-term support and encouragement. During the writing of this paper, the author benefits a lot from many helpful discussions with him.

\end{document}